\documentclass[12pt,reqno]{amsart}

\usepackage{amssymb,enumerate,amsthm,multicol,graphicx,amsmath}

\usepackage{pdfpages}

\DeclareMathAlphabet{\E}{U}{eus}{m}{n}     
    
\renewcommand{\a}{\alpha}
\newcommand{\Z}{{\E{Z}}}

\newcommand{\V}{{\mathcal V}}      

\newcommand{\A}{{\mathcal A}}

\newcommand{\PP}{{\mathbb P}}

\newcommand{\kk}{{\Bbbk}}

\newcommand{\LL}{{\mathfrak L}}

\newcommand{\p}{\mathfrak p}

\newtheorem{thm}{Theorem}[section]

\newtheorem{cor}[thm]{Corollary}
\newtheorem{prop}[thm]{Proposition}

\theoremstyle{definition}

\newtheorem{defn}[thm]{Definition}

\newtheorem{rmk}[thm]{Remark}
\newtheorem{rmks}[thm]{Remarks}

\newtheorem{ack}{Acknowledgments\!\!}  

\newtheorem{Pf}{Proof$\!\!$}         
         
\newenvironment{pf}{\begin{Pf}}{\qed\end{Pf}}

\DeclareMathSymbol{\twoheadrightarrow}  {\mathrel}{AMSa}{"10}

\newcounter{letter}
\renewcommand{\theletter}{\rom{(}\alph{letter}\rom{)}}

\newcounter{rnum}
\renewcommand{\thernum}{\rom{(}\roman{rnum}\rom{)}}

\begin{document}
\baselineskip21pt


\title[Line Scheme of a Family of Quadratic Quantum $\PP^3$s]%
{The One-Dimensional Line Scheme of a\\[2mm]
Family of Quadratic Quantum $\PP^3$s}


\keywords{line scheme, point scheme, elliptic curve, %
regular algebra, Pl\"ucker coordinates. \newline
\indent This work was supported in part by NSF grants DMS-0900239 and
DMS-1302050}

\maketitle

\vspace*{0.1in}

\baselineskip15pt

\begin{center}
\begin{tabular}[t]{c}
\textsc{Derek Tomlin}\\
derek.tomlin@mavs.uta.edu\\

\end{tabular}
\quad  and \quad  
\begin{tabular}[t]{c}
\textsc{Michaela Vancliff}\\
vancliff@uta.edu\\
www.uta.edu/math/vancliff
\end{tabular}\\[3mm]
Department of Mathematics, P.O.~Box 19408\\
University of Texas at Arlington,\\
Arlington, TX 76019-0408
\end{center}

\setcounter{page}{1}
\thispagestyle{empty}

\bigskip
\bigskip

\begin{abstract}
\baselineskip15pt

The attempted classification of regular algebras of global dimension four, so-called quantum $\PP^3$s, has been a driving force for modern research in noncommutative algebra.  Inspired by the work of Artin, Tate, and Van den Bergh, geometric methods via schemes of \emph{d-linear} modules have been developed by various researchers to further their classification. In this work, we compute the line scheme of a certain family of algebras whose generic member is a candidate for a generic quadratic quantum $\PP^3$. We find that, viewed as a closed subscheme of $\PP^5$, the generic member has a one-dimensional line scheme consisting of eight curves: one nonplanar elliptic curve in a $\PP^3$, one nonplanar rational curve with a unique singular point, two planar elliptic curves, and two subschemes, each consisting of the union of a nonsingular conic and a line.
 
\end{abstract}

\baselineskip18pt

\bigskip
\bigskip


\section*{Introduction}

Regular algebras of global-dimension $n$ are often viewed as noncommutative analogues of polynomial rings on $n$ variables and will herein be called quantum $\PP^{n-1}$s.  Many geometric objects (points, lines, etc.) associated with polynomial rings have counterparts that are associated with quantum $\PP^{n}$s.  In \cite{ATV1}, quantum $\PP^2$s were classified via their point schemes,  and subsequently, a similar description of quantum $\PP^3$s is desired.  Using the definitions given in \cite{SV1}, a classification of quantum $\PP^3$s using their point schemes or their line schemes is sought.  This is due to the well-known fact that if a generic quadratic quantum $\PP^3$ exists, then it has a point scheme consisting of exactly twenty distinct points and a one-dimensional line scheme (see \cite[\S 1D]{Vmsri}).

To date, not many line schemes of quadratic quantum $\PP^3$s are known, especially of those considered to be candidates for generic quadratic quantum $\PP^3$s.  Fortunately, the regular graded skew Clifford algebras introduced in \cite{CV1} produce many examples of algebras that are candidates for generic quantum $\PP^3$s.  With the technique introduced in \cite{SV2} for computing the line scheme of any quadratic algebra on four generators, these algebras promise to shed light on the open question of classifying all quantum $\PP^3$s.  In this article, we compute and analyze the line scheme of a family of algebras introduced in \cite[\S 5, Example 2]{CV1} whose generic member is a candidate for a generic quadratic quantum $\PP^3$. The line scheme of these algebras has not been identified elsewhere, not even in \cite{CV1}.

In \cite{ChV}, the line scheme of a different algebra from \cite[\S 5]{CV1} was computed, and a conjecture was given in \cite{ChV} regarding the line scheme of the most generic quadratic quantum $\PP^3$. Although the line scheme described herein is different from the one computed in \cite{ChV}, our results nevertheless support \cite[Conj. 4.2]{ChV}.

The article is outlined as follows. We introduce the algebras of study in Section \ref{sec1}, and we compute their point schemes in Section \ref{sec2} in Theorem \ref{ptscheme}. In particular, we verify that each algebra has a point scheme consisting of exactly twenty distinct points. Section \ref{sec3} is dedicated to the computation and description of the line scheme viewed as a subscheme of $\PP^5$, while Section \ref{sec4} discusses the lines in $\PP^3$ that are parametrized by the line scheme. With Theorems \ref{8curves} and \ref{redthm}, we prove that the line scheme of the generic member is reduced and is the union of eight irreducible curves: a nonplanar elliptic curve in a $\PP^3$ (spatial elliptic curve), a nonplanar rational curve in a $\PP^3$ with one singular point, two planar elliptic curves, and two subschemes, each consisting of the union of a nonsingular conic and a line. Unlike \cite{ChV}, we investigate the intersection points of the components of the line scheme and consider how they relate to information in the algebra in Corollary \ref{intpoints} and Remarks \ref{midlrmks}. In Theorem \ref{6lines}, we show that exactly sixteen points of the point scheme lie on only finitely many lines parametrized by the line scheme. The polynomials used for calculations throughout the article are listed in the Appendix in Section~\ref{app}.


\bigskip
\bigskip

\section{The Algebras} \label{sec1}

In this section, we introduce the algebras that will be studied within this paper. These algebras depend on a scalar $\a \in\kk$; the scalar $\a$ will be called generic if it satisfies $\a(1~-~\a^2)~\neq~0$.

Throughout, $\kk$ denotes an algebraically closed field, 
and for a vector space $V$, we denote its vector-space dual as $V^{\text{*}}$. In the first two sections, we assume char($ \kk$) $ \neq 2$; however in later sections, we take char($ \kk$) $ = 0$ due to computations therein.

\begin{defn}\label{Aa}\cite{CV1}
Let $\A(\a)$ denote the $\kk$-algebra
on degree-one generators $x_1, \ldots, x_4$ with the following defining relations:
\[
\begin{array}{rclrclrcl}
 x_3 x_1+x_1 x_3 & = & 0, \quad & \quad x_3 x_2 - x_2 x_3& = & 0, \quad & \quad
2x_2^2+\a x_3^2 & = & x_1^2,\\[3mm]
x_4 x_1 +x_1x_4& = & 0,  & x_2^2-x_4^2 & = & 0, &  x_4x_2+x_2x_4 & = & x_3^2,
\end{array}
\]
where $\a\in\kk$ is generic.
	\end{defn}

\noindent The algebra $\A(\a)$ is a member of the family of algebras $A(\a_1, \, \a_2, \, \beta_1, \, \beta_2 \,)$ presented in Example~2 of \cite[\S 5]{CV1}; in that context, $\A(\a)=A(\a, \, 0, \, 2, \, 0 \,)$.  The algebras in \cite[\S 5, Example 2]{CV1} were constructed with the objective of generating a generic quadratic regular algebra of global dimension four -- a so-called generic quantum $\PP^3$.  By construction in \cite{CV1}, $\A(\a)$ is a noetherian domain having Hilbert series the same as that of the polynomial ring on four variables. Furthermore, $\A(\a)$ is a candidate for a generic quantum $ \PP^3 $ since it has a finite point scheme consisting of exactly twenty distinct points and a one-dimensional line scheme (cf. \cite[\S 1D]{Vmsri}).


\bigskip

\section{The Point Scheme of $\A(\a)$}\label{sec2}
In this section, we compute the point scheme of $\A(\a)$.  The method used follows that of \cite{ATV1}, and we continue to assume char$( \kk) \: \neq 2$ in this section.

Let $ V = \sum_{i=1}^4 \kk x_i$, and write the relations of $\A(\a)$ as the product $Mx$, where $x$ is the column vector given by $x^T=(x_1,\ldots, x_4)$ and $M$ is the $6\times 4$ matrix:

\[
\begin{bmatrix} 
x_3 & 0 & x_1 & 0\\
0 & -x_3 & x_2 & 0\\
x_4 & 0 & 0 & x_1\\
-x_1 & 2 x_2 & \a x_3 & 0\\
0 & x_4 & -x_3 & x_2\\
0 & x_2 & 0 & -x_4
\end{bmatrix}.
\]

\noindent Let $\Gamma$ be the scheme in $\PP(V^\text{*})\times \PP(V^\text{*})$ that is the zero locus of the six defining relations of $\A(\a)$, and let $\p$ denote the image of $\Gamma$ under the projection map from $\PP(V^\text{*})\times \PP(V^\text{*})$ to the first copy of $\PP(V^\text{*})$.  The scheme $\p$ is the zero locus of the $4\times 4$ minors of $M$, which are the polynomials given in the Appendix in Section \ref{app1}.  Referencing \cite{SV}, $\A(\a)$ satisfies sufficient conditions such that $\Gamma$ is the graph of an automorphism $\sigma\in \text{Aut}(\p)$. Moreover, the \textit{point scheme}, $\p(\a)$, of $\A(\a)$ may be viewed as $\Gamma$; the closed points of which parametrize the isomorphism classes of point modules over $\A(\a)$.  We will show, for generic $\a$, that the scheme $\p(\a)$ is finite with exactly twenty distinct points and that $\sigma\in \text{Aut}(\p)$ has ten orbits of order two.

Let $q=( \lambda_1, \ldots ,\lambda_4 ) \in \p$.  If $\lambda_4=0$, it follows that $q$ is one of the points $e_1=(1, \, 0, \, 0, \, 0 )$, $e_2=( 0, \, 1, \, 0, \, 0 )$, or $e_3=(0, \, 0, \, 1, \, 0)$.  We henceforth assume $\lambda_4=1$. By additionally assuming that $\lambda_2=0$, it can be shown that $q$ must equal $e_4=(0, \, 0, \, 0, \, 1)$.  Thus, we proceed with $\lambda_4=1$ and $\lambda_2\neq 0$.

Under these restrictions, Polynomial \ref{app1}.\ref{pp3} becomes $\;x_1x_3(x_2-1)(x_2+1)$, and hence it suffices to consider exactly four cases: $\lambda_2=1$, $\lambda_2=-1$, $\lambda_1=0$ and $\lambda_3=0$.

If $\lambda_2=1$, then Polynomials \ref{app1}.\ref{pp12} and \ref{app1}.\ref{pp13} imply that $\,x_1^2+2(1+\a)=0$ and $x_3^2=2$.  Moreover, the remaining polynomials vanish on the zero locus of these two polynomials; this case therefore yields a subvariety, $\Z_1$, of $\p$ consisting of exactly four distinct points.   Similarly for $\lambda_2=-1$, the zero locus, $\Z_2$, of $\,x_1^2+2(1-\a)$ and $x_3^2+2$ characterizes another affine subvariety of $\p$ that contains exactly four distinct points.

Assuming that $\lambda_3=0$ implies that $q\in\p$ if and only if $q$ belongs to the zero set, $\Z_3$, of $\,x_1^2+2$ and $x_2^2+1$. Lastly, by setting $\lambda_1=0$ we see that rank$(M)=0$ if and only if $q$ is in the zero locus, $\Z_4$, of $\, \a x_2^2+2x_2+\a$ and $\a x_3^2+2x_2^2$.   In particular,  $\Z_3$ and $\Z_4$ each contain exactly four distinct points (since char$(\kk)\neq 2$ and $\a$ is generic).
  
The following remark and the above discussion provide the proof of Proposition \ref{ptS}.

\begin{rmk}\label{rmk} (cf. \cite[\S 1D]{Vmsri})
If the zero locus, $\mathfrak{z}$, of the defining relations of a quadratic algebra on four 
generators with six defining relations is finite, then $\mathfrak{z}$ consists of twenty
points counted with multiplicity.
\end{rmk}

\begin{prop}\label{ptS}
Let $\A(\a)$, $\p$, and $\Z_1, \ldots, \Z_4$ be as above.  Writing $\Z_0=\{ e_1,\ldots,e_4 \}$, 
the scheme $\p=\bigcup_{i=0}^4\Z_i$, and $\p$ contains exactly twenty distinct points, each of multiplicity one.

\end{prop}
\begin{pf}
The above discussion proves that the scheme $\p$ is finite and has exactly twenty distinct closed points, namely: 
\begin{enumerate}
\itemsep1mm
\item[\textup{(o)}] $ \Z_0=\{ e_1,\ldots,e_4 \} $,

\item[\textup{(i)}]$ \Z_1=\{ \,(\lambda_1,\,1,\,\lambda_3,\,1)\;\;\;\;| \; \lambda_1^2=-2(1+\a) \, , \; \lambda_3^2=2\, \},$

\item[\textup{(ii)}]$ \Z_2=\{ \,(\lambda_1,\,-1,\,\lambda_3,\,1)\;| \; \lambda_1^2=-2(1-\a) \, ,  \; \lambda_3^2=-2\, \},$

\item[\textup{(iii)}]$\Z_3= \{ 
\, (\lambda_1,\,\lambda_2,\,0,\,1)\;\;\;\; | \; \lambda_1^2=-2,\,  \; \lambda_2^2=-1\,\}, $

\item[\textup{(iv)}]$\Z_4=\{ \, (0,\,\lambda_2,\,\lambda_3,\,1)\;\;\;\;|\; \a\lambda_2^2+2\lambda_2+\a=0, \,  \; \a\lambda_3^2=-2\lambda_2^2\,\}.$
\end{enumerate}
Paired with Remark 2.1, the proof is complete. \end{pf}

\begin{cor}\label{ptscheme}
The points of $\p(\a)$ are closed, of the form $(q,\,\sigma(q)) \in \Gamma \subset \PP(V^\text{*})\times\PP(V^\text{*})$, and are given by

\begin{enumerate}
\itemsep1mm
\item[\textup{(o)}] $ (e_1,\,e_2),\,(e_2,\,e_1)\,,(e_3,\,e_4),\,(e_4,\,e_3)$ ,

\item[\textup{(i)}]$ (\;(\lambda_1,\,1,\,\lambda_3,\,1), \,(-\lambda_1,\,1,\,\lambda_3,\,1) \;)  , $ where $\, \lambda_1^2=-2(1+\a) \, $ and $ \lambda_3^2=2 $, 

\item[\textup{(ii)}]$ ( \;
(\lambda_1,\,-1,\,\lambda_3,\,1),\, (-\lambda_1,\,-1,\,\lambda_3,\,1)\;) , $ where $
 \, \lambda_1^2=-2(1-\a) \, $ and $ \lambda_3^2=-2 $,

\item[\textup{(iii)}]$( \; 
 (\lambda_1,\,\lambda_2,\,0,\,1),\,  (-\lambda_1,\,-\lambda_2,\,0,\,1)
 \; ), $ where $ \, 
 \lambda_1^2=-2 \, $ and $ \lambda_2^2=-1 $,

\item[\textup{(iv)}] $ ( \; 
(0,\,\lambda_2,\,\lambda_3,\,1),\,  (0,\,\lambda_2,\,\lambda_3,\,\lambda_2^2)\;)
, $ where $ \,
\a\lambda_2^2+2\lambda_2+\a=0 \, $ and $ \a\lambda_3^2=-2\lambda_2^2 $.
 
\end{enumerate}
Furthermore, $\sigma(\Z_i)=\Z_i$ for every $i=0,\ldots,4$, and $\sigma$ has ten orbits of order two.
\end{cor}

\begin{proof}
The result follows from Proposition \ref{ptS} and by computation using the polynomials in the Appendix in Section \ref{app1}. 
\end{proof}


\section{The Line Scheme of $\A(\a)$}\label{sec3}
In this section, we compute the line scheme, $\LL(\a)$, of $\A(\a)$ following the
method introduced in \cite{SV2}.  The method used is summarized in Section
\ref{meth} while being applied to $\A(\a)$.  The closed points of the line scheme
are computed in Section \ref{clines}, and, following that discussion, we prove that the scheme $\LL(\a)$ is a reduced scheme and is therefore described by its closed points.  

Henceforth, we assume char$(\kk)=0$.  Recall the definition $ V = \sum_{i=1}^4 \kk x_i$ given at the start of Section \ref{sec2}.

\subsection{Methodology}\label{meth} \hfill

In \cite{SV2}, Shelton and Vancliff introduced a method for computing the line scheme of any quadratic algebra on four generators that is a domain having Hilbert series $(1-t)^{-4}$.  In this subsection, we summarize their method while applying it to $\A(\a)$.

To describe $\LL(\a)$, one first constructs $\A(\a)^!$, the Koszul dual of $\A(\a)$.  Since $\A(\a)$ has four generators and six quadratic defining relations, $\A(\a)^!$  is a quadratic algebra on four generators, $\{z_1,\ldots,z_4\}$, with ten defining relations, $\{f_1,\ldots,f_{10}\}$.  In particular, $f_i(r)=0$ for all elements~$r$ in the span of the defining relations of $\A(\a)$, for all $i=1,\ldots,10$, and one takes $ \{z_1,\ldots,z_4\}$ to be the dual basis in $V^\text{*}$ of $\{x_1,\ldots,x_4\}$. Similar to the start of Section \ref{sec2}, one presents the defining relations of $\A(\a)^!$ by a matrix equation $\hat{M}z=0$, where $z^T=(z_1,\ldots,z_4)$ and $\hat{M}$ is a $10\times 4$ matrix whose entries are linear forms in the $z_i$.

Continuing to follow \cite{SV2}, a $10\times 8$ matrix is produced by concatenating two $10\times 4$ matrices created from $\hat{M}$.  The first such matrix is obtained by replacing each $z_i$ in $\hat{M}$ with $u_i\in\kk$, where 
$\kk\sum_{i=1}^4 u_i x_i
\in\PP(V)$, and the second matrix is obtained by replacing each $z_i$ in $\hat{M}$ with
$v_i\in\kk$, where 
$\kk\sum_{i=1}^4 v_i x_i
\in\PP(V)$. Applying this process to $\A(\a)$ gives the following $10\times 8$ matrix: 

\[
\small
\mathcal M(\a) = 
\begin{bmatrix}
\a u_1 & u_4 &  u_3 & 0 & \a v_1 &  v_4 &  v_3 & 0\\
0 & u_1 & 0 & 0 & 0 & v_1 & 0 & 0\\
-u_3 & 0 & u_1 & 0 &-v_3 & 0 & v_1 & 0 \\
-u_4 & 0 & 0 & u_1 & -v_4 & 0 & 0 & v_1 \\
u_2 & 0 & 0 & 0 & v_2 & 0 & 0 & 0\\
0 & \a u_2 -2 u_4 & -2 u_3 & \a u_4 & 0 & \a v_2 - 2 v_4 & - 2 v_3 & \a v_4 \\
0 & u_3 & u_2 & 0 & 0 & v_3 & v_2 & 0\\
0 & -u_4 & 0 & u_2 & 0 & -v_4 & 0 & v_2\\
0 & 0 & 0 & u_3 & 0 & 0 & 0 &  v_3\\
0 & 0 & u_4 & 0 & 0 & 0 & v_4 & 0
\end{bmatrix}.
\]
\quad
Each of the forty-five $8\times 8$ minors of $\mathcal M(\a)$ is a bihomogeneous polynomial of bidegree $(4,4)$ in the $u_i$ and $v_i$.  In particular, each $8\times 8$ minor is a linear combination of products of polynomials of the form $N_{ij}=u_iv_j-u_jv_i$, where $1\leq i<j\leq 4$. Thus, $\mathcal M(\a)$ produces forty-five quartic polynomials in the six variables $N_{12},\ldots,N_{34}$.  Applying the orthogonality isomorphism,
\\[-2mm]
\[
\begin{array}{l}
N_{12}\mapsto M_{34}, \quad\ N_{13}\mapsto -M_{24},\quad \ N_{14}\mapsto M_{23},\\[3mm]
N_{23}\mapsto M_{14},\quad \ N_{24}\mapsto -M_{13}, \quad\ N_{34}\mapsto M_{12},
\end{array}
\]
\quad\\[-1mm]
described in \cite{SV2} to these polynomials yields forty-five quartic polynomials in the Pl\"ucker coordinates, $M_{12},\,M_{13},\,M_{14},\,M_{23},\,M_{24},$ and $M_{34},$ on $\PP^5$.  

The line scheme, $\LL(\a)$, of $\A(\a)$ can be realized as the scheme of zeros in $\PP^5 $ of these forty-five quartic polynomials in the  $M_{ij}$ coordinates coupled with the Pl\"ucker polynomial $P=M_{12}M_{34}-M_{13}M_{24}+M_{14}M_{23}$.  These polynomials for $\A(\a)$ were computed using Wolfram's Mathematica and are listed in the Appendix in Section \ref{app2}.  

For the duration of this section, we outline the computation of, and describe the closed points of, $\LL(\a)$ as a subscheme of $\PP^5$.  The lines in $\PP(V^\text{*})$ characterized by these closed points will be discussed in the following section.
\bigskip
\subsection{The Closed Points of $\LL(\a)$}\label{clines}\hfill

We compute the closed points of the line scheme $\LL(\a)$ in this subsection. Following this computation, we show that $\LL(\a)$ is reduced and hence given by its closed points.  Let $\LL'(\a)$ denote the variety of closed points of $\LL(\a)$, and for a set $S$ of polynomials, we write $\V(S)$ for its zero locus.

Using Mathematica to compute a Gr\"obner basis of the forty-six polynomials in Section \ref{app2} produces a list of polynomials that contains $M_{13}M_{23}^2M_{24}$.  If $M_{13}=M_{23}=M_{24}=0$, then $M_{12}=0$ or $M_{34}=0$.  If $M_{12}=0=M_{34}$, then a unique solution is obtained. If $M_{34}\neq 0$, then $M_{12}=0$, and only polynomials \ref{app2}.\ref{lp27} and \ref{app2}.\ref{lp45} remain; that is, $M_{14}M_{34}\left(M_{14}^2+\a M_{34}^2\right)=0$ and $M_{34}^2\left(M_{14}^2+\a M_{34}^2\right)=0$, giving two distinct solutions.  Similarly, if $M_{12}\neq 0$, then $M_{34}=0$, and the only surviving polynomials are \ref{app2}.\ref{lp36} and \ref{app2}.\ref{lp39}. Whence, $M_{12}M_{14}\left(M_{12}^2+M_{14}^2\right)=0$ and $M_{12}^2\left(M_{12}^2+M_{14}^2\right)=0$, which also yields two distinct solutions.  Thus at this stage, we have found five distinct points of $\LL'(\a) $; namely 
\[
\begin{array}{l}

\V\left(M_{12},\, M_{13},\,M_{23},\,M_{24},\,M_{34}\right),\qquad \V\left(M_{12},\, M_{13},\,M_{23},\,M_{24},\,M_{14}\, \pm i a M_{34}\right), \\[2mm]

\qquad\qquad\text{and}\quad\quad \;\;\V\left(M_{13},\, M_{23},\,M_{24},\,M_{34},\,M_{12} \pm i M_{14}\right),

\end{array}
\]
\noindent where $i^2=-1$ and $a^2=\a$.  It follows that six cases remain to be considered:\\[-2mm]

\begin{center}
\begin{tabular}{rlrl}
(I)& $M_{23} \neq 0, M_{13}=0=M_{24}$,\hspace*{9mm} &
(II) & $M_{13}M_{24} \neq 0, M_{23}=0$,\\[2mm]
(III)& $M_{13}\neq 0, M_{23}=0=M_{24}$, &
(IV)& $M_{13}M_{23}\neq 0, M_{24}=0$, \\[2mm]
(V) & $M_{24} \neq 0, M_{13}=0=M_{23}$, &
(VI) & $M_{23}M_{24} \neq 0, M_{13}=0$.
\end{tabular}
\end{center}

\quad\\

\noindent
\textbf{Case (I):}
$M_{23}\neq 0$ and $M_{13}=0= M_{24}$.\\
With these assumptions, a Gr\"obner basis contains $M_{34}^2\rho$, where $\rho=M_{12}^2 + M_{14}^2 + \a M_{23}^2 - 2 M_{23} M_{34} +\a M_{34}^2$.  If $M_{34}=0$, it follows that $M_{14}=0$, $\,M_{12}^2(M_{12}^2+\a M_{23}^2)=0$, and the remaining polynomials vanish, in which case the zero set consists of three distinct points that are contained in the subvariety $\LL_{6a}\cup\LL_{6b}$ discussed in Case (VI) below.  Thus, we may assume $M_{34}\neq 0$; in particular, this means $\rho=0$.  Computing another Gr\"obner basis yields a set of four polynomials, one of which is the image of the Pl\"ucker polynomial $P$ and one of which is~$\rho$.  Using the image of $P$ to substitute for $M_{14}M_{23}$ in each of the other polynomials in this Gr\"obner basis reveals that they each vanish if and only if $\rho=0$.  This case therefore provides the irreducible component
\[
\LL_1=\V\left(M_{13},\,M_{24},\,M_{12}M_{34}+M_{14}M_{23},\,M_{12}^2 + M_{14}^2 + \a M_{23}^2 - 2 M_{23} M_{34} +\a M_{34}^2\right)
\]
\noindent of $\LL'(\a)$.\\[1mm]

\noindent
\textbf{Case (II):}
$M_{13}M_{24}\neq 0$ and $M_{23}=0$.

\noindent In this case, a Gr\"obner basis contains $M_{13}M_{14}M_{24}^2$, implying  $M_{14}=0$. Imposing this restriction implies $M_{13}M_{24}\phi=0$, where $\phi=M_{12}^2+2M_{24}^2+\a M_{34}^2$. As $M_{13}M_{24}\neq 0$ by assumption, it follows that $\phi=0$. With these constraints, there is a Gr\"obner basis consisting of exactly four polynomials, one of which is $\phi$ and one of which is the image of $P$. Similar to Case (I), by using this image to substitute for $M_{13}M_{24}$ in each of the other polynomials, we see that they all vanish if and only if $\phi=0$.  This case thus gives the irreducible component 
\[
\LL_2=\V\left(M_{14},\,M_{23},\,M_{12}M_{34}-M_{13}M_{24},\,M_{12}^2+2M_{24}^2+\a M_{34}^2\right)
\]
\noindent of $\LL'(\a)$.\\[1mm]

\noindent\textbf{Case (III):} $M_{13}\neq 0 $ and $M_{23}=0=M_{24}$.

\noindent Under the hypotheses for this case, \ref{app2}.\ref{lp42} becomes $M_{13}M_{34}^3$, and so $M_{34}=0$.  Enforcing this condition leaves only three polynomials: $M_{1j}\left(M_{12}^3 - M_{13}^2 M_{14} + M_{12} M_{14}^2\right)$ for $j=2,3,4$.  Noting again that $M_{13}\neq 0$, we find that an irreducible component of $\LL'(\a)$ is described by 
\[
\LL_3=\V\left(M_{23},\,M_{24},\,M_{34},\, M_{12}^3 - M_{13}^2 M_{14} + M_{12} M_{14}^2 \right).
\]
\\[1mm]
\noindent\textbf{Case (IV):} $M_{13}M_{23}\neq 0$ and $M_{24}=0$.

\noindent By \ref{app2}.\ref{lp24}, these assumptions imply $M_{12}=0$, so by \ref{app2}.\ref{lp0}, $M_{14}=0$.  Applying these restrictions to the remaining polynomials in Section \ref{app2} yields a set of polynomials in which every element has 
\[ 
M_{13}^2 M_{23} - \a M_{23}^2 M_{34} + 2 M_{23} M_{34}^2 - \a M_{34}^3
\]
\noindent as a factor.  As one of the multipliers is $M_{13}$, this case describes the irreducible component 
\[
\LL_4=\V\left(M_{12},\,M_{14},\,M_{24},\,M_{13}^2 M_{23} - \a M_{23}^2 M_{34} + 2 M_{23} M_{34}^2 - \a M_{34}^3\right)
\]
\noindent of $\LL'(\a)$.\\[1mm]

\noindent\textbf{Case (V):} $M_{24} \neq 0$ and $ M_{13}=0=M_{23}$

\noindent With these assumptions, a Gr\"obner basis contains $M_{12}M_{34}$, and each element of this basis has either $M_{12}$ or $M_{34}$ as a factor.  It follows that, if $M_{12}=0=M_{34}$, then the line
\[
\LL_{5a}=\V\left(M_{12},\,M_{13},\,M_{23},\,M_{34}\right)
\]
\noindent is an irreducible component of $\LL'(\a)$.

Assuming $M_{12}\neq 0 $ forces $M_{34}$ to be zero, which by \ref{app2}.\ref{lp2}, implies that $M_{14}=0$.  By \ref{app2}.\ref{lp31} and \ref{app2}.\ref{lp39}, we have $M_{12}^2+2M_{24}^2=0$, in which case the zero set consists of two distinct points that belong to $\LL_{6b}$ in Case (VI) below.  If instead, $M_{34}\neq 0$ so that  $M_{12}=0$, then the zero locus is given by \ref{app2}.\ref{lp27} and \ref{app2}.\ref{lp45}, namely
\[
 M_{14}M_{34}\left(M_{14}^2+2M_{24}^2+\a M_{34}^2\right) \; = \;0\; = \; M_{34}\left(M_{14}^2+2M_{24}^2+\a M_{34}^2\,\right),
\] 
\noindent and so we obtain the irreducible component 
\[
\LL_{5b}=\V\left(M_{12},\,M_{13},\,M_{23},\,M_{14}^2+2M_{24}^2+\a M_{34}^2\right)
\]

\noindent of $\LL'(\a)$.  This component is a nonsingular conic since $M_{14}^2+2M_{24}^2+\a M_{34}^2$ is a rank-three quadratic form.  Both $\LL_{5a}$ and $\LL_{5b}$ lie in the plane $\V\left(M_{12},\,M_{13},\,M_{23}\right)$ and meet in two distinct points.\\[1mm]

\noindent\textbf{Case (VI):} $M_{23}M_{24} \neq 0$ and $ M_{13}=0$.

\noindent With these assumptions, \ref{app2}.\ref{lp11} implies that $M_{23}M_{24}^2 M_{34}=0$, and hence $M_{34}=0$.  Thus, using \ref{app2}.\ref{lp0}, we find that $M_{14}=0$.  Making these assignments produces exactly three polynomials: $M_{12}^2\psi,\, M_{12}M_{23}\psi$ and $M_{12}M_{24}\psi$, where $\psi=M_{12}^2+\a M_{23}^2+2 M_{24}^2$ (from \ref{app2}.\ref{lp39}, \ref{app2}.\ref{lp37}, and \ref{app2}.\ref{lp31} respectively).  It follows that this case yields two irreducible components of $\LL'(\a)$, namely
\[
\LL_{6a}=\V\left(M_{12},\,M_{13},\,M_{14},\,M_{34}\right),
\]
\noindent which is a line, and 

\[
\LL_{6b}=\V\left(M_{13},\,M_{14},\,M_{34},\,M_{12}^2+\a M_{23}^2+2 M_{24}^2\right),
\]
\noindent which is a nonsingular conic since $M_{12}^2+\a M_{23}^2+2 M_{24}^2$ is a rank-three quadratic form.  Both $\LL_{6a}$ and $\LL_{6b}$ lie in the plane $\V\left(M_{13},\,M_{14},\,M_{34}\right)$ and meet in two distinct points. \\[1mm]

We summarize the above discussion in the next result.  The reader should note that, in the following, ``spatial curve" indicates a nonplanar curve in $\PP^5$ that is contained in a linear subscheme of $\PP^5$ that is isomorphic to $\PP^3$.

\begin{thm}\label{8curves}
Let $\LL'(\a)$ denote the reduced variety of the line scheme, $\LL(\a)$, of $\A(\a)$.  For generic $\a$, $\LL'(\a)$ is the union of eight irreducible components:

\begin{enumerate}
\item[\textup{(I)}] $\LL_1 =\V\left(M_{13},\,M_{24},\,M_{12}M_{34}+M_{14}M_{23},\,M_{12}^2 + M_{14}^2 + \a M_{23}^2 - 2 M_{23} M_{34} +\a M_{34}^2\right),$ 
	      which is a spatial elliptic curve. 
\item[\textup{(II)}] $\LL_2  = \V\left(M_{14},\,M_{23},\,M_{12}M_{34}-M_{13}M_{24},\,M_{12}^2+2M_{24}^2+\a M_{34}^2\right),$
              which is a spatial rational curve with one singular point.
\item[\textup{(III)}] $\LL_3=\V\left(M_{23},\,M_{24},\,M_{34},\, M_{12}^3 - M_{13}^2 M_{14} + M_{12} M_{14}^2 \right),$ 
              which is a planar elliptic curve.
\item[\textup{(IV)}] $\LL_4=\V\left(M_{12},\,M_{14},\,M_{24},\,M_{13}^2 M_{23} - \a M_{23}^2 M_{34} + 2 M_{23} M_{34}^2 - \a M_{34}^3\right),$ 
              which is a planar elliptic curve.
\item[\textup{(Va)}] $\LL_{5a} = \V\left(M_{12},\,M_{13},\,M_{23},\,M_{34}\right),$ 
              which is a line.
\item[\textup{(Vb)}] $\LL_{5b} =\V\left(M_{12},\,M_{13},\,M_{23},\,M_{14}^2+2M_{24}^2+\a M_{34}^2\right),$ 
              which is a nonsingular conic.              
\item[\textup{(VIa)}] $\LL_{6a} = \V\left(M_{12},\,M_{13},\,M_{14},\,M_{34}\right),$ 
	      which is a line.
\item[\textup{(VIb)}] $\LL_{6b} = \V\left(M_{13},\,M_{14},\,M_{34},\,M_{12}^2+\a M_{23}^2+2 M_{24}^2\right),$ 
	      which is a nonsingular conic.
\end{enumerate}

\end{thm}

\begin{pf}
 The polynomials describing $\LL'(\a)$ are given in the preceding work, as is the geometric description of $\LL_{5a},\,\LL_{5b},\,\LL_{6a}$ and $\LL_{6b}$.  Here, we give the geometric description of the remaining components.  Observe that the five distinct points discussed at the start of Section \ref{clines} are contained in $\LL_{5a}$, $\LL_{5b} $, and $\LL_1$. 

(I) Write $f_1=M_{12}M_{34}+M_{14}M_{23}\,$ and $\,g_1=M_{12}^2+M_{14}^2+\a M_{23}^2-2M_{23}M_{34}+\a M_{34}^2$ viewed in the homogeneous coordinate ring $\kk[M_{12}, \, M_{14}, \, M_{23}, \, M_{34}]$ of $\PP^3$.  Since $\a$ is generic, the Jacobian matrix

\[
\footnotesize
J_1=
\begin{bmatrix}
M_{34} & M_{23} & M_{14} & M_{12} \\[3mm]
2 M_{12}& 
2 M_{14}& 
2( \a M_{23}-M_{34} )& 
2( \a M_{34}-M_{23} )  
\end{bmatrix}
\]
\quad\\[-1mm] of the system $\{f_1,\,g_1\}$ has rank two on $ \V(f_1,g_1)$. Hence, $\V(f_1,\,g_1)\subset\PP^3$ is nonsingular and reduced. Referencing the proofs of \cite[Proposition 2.5]{Smith.Staff} and \cite[Theorem 3.1]{ChV}, if $\V(f_1,\,g_1)$ is not irreducible, then there exists a point in the intersection of two of its irreducible components and $J_1$ would have rank at most one at that point, which is false.  Thus, $\V(f_1,\,g_1)$ is a reduced, nonsingular irreducible elliptic curve, and the same is true for $\LL_1$.

(II) Let $f_2=M_{12}M_{34}-M_{13}M_{24}$ and $g_2=M_{12}^2+2M_{24}^2+\a M_{34}^2$ viewed in the homogeneous coordinate ring $\kk[M_{12}, \, M_{13}, \, M_{24}, \, M_{34}]$ of $\PP^3$.  The Jacobian matrix of the system $\{f_2,\,g_2\}$ is

\[
\footnotesize
J_2=
\begin{bmatrix}
M_{34} & -M_{24} & -M_{13} & M_{12} \\[3mm]
2M_{12}& 
0 & 
4M_{24}& 
2\a M_{34} 
\end{bmatrix},
\]
\quad\\[-1mm] and $\varepsilon_2=\V(M_{12},\,M_{24},\,M_{34})$ is the unique point
in $\V(f_2,\,g_2)$ at which rank$(J_2)<2$; thus, $\varepsilon_2$ is the unique
singular point of the curve $\V(f_2,\,g_2)\subset\PP^3$.  Referencing
\cite{Smith.Staff} and \cite{ChV} as in Case~(I), 
if $\V(f_2,\,g_2)$ is not irreducible, then at least two of its components intersect at points $p$ such that rank$(J_2|_p)\leq 1$;  the above discussion implies that such components intersect at $\varepsilon_2$. We localize $\V(f_2,\,g_2)$ at $\varepsilon_2$, and, in so doing, we write $x=M_{12}/M_{13}$, $y=M_{24}/M_{13}$, and $z=M_{34}/M_{13}$.  This process gives $f_2|_{M_{13}=1}=\hat{f}= xz-y\;$ and $\;g_2|_{M_{13=1}}=\hat{g}=x^2+2y^2+\a z^2$. Substituting for~$y$ yields
 \[
 \V(\hat{f},\,\hat{g})=\V(x^2+2x^2z^2+\a z^2)\subset \mathbb{A}^2 . 
 \]
\noindent In \cite{DT} it is shown that $\V(\hat{f},\,\hat{g})$ is irreducible in $\mathbb{A}^2$, and so only one component passes through~$\varepsilon_2$.  It follows that $\V(f_2,\,g_2)$ is irreducible in $\PP^3$.  The multiplicity of $\V(f_2,\,g_2)$ at $\varepsilon_2$ is given by $\text{deg}(x^2+ \a z^2)=2$, and hence $\varepsilon_2$ is a double point with two distinct tangent lines. In particular, $\V(f_2,\,g_2)$ is a curve on the rank-three quadric
\[
\V\left(M_{12}^2+2M_{24}^2+\a M_{34}^2\right)\subset \PP^3,
\]
\noindent and it self-intersects at exactly one point, $\varepsilon_2$, which is the unique singular point of $\V\big(M_{12}^2+2M_{24}^2+\a M_{34}^2\big)$.  Hence, $\LL_2$ is an irreducible degree-four curve with unique singular point $E_2=(0,\,1,\,0,\,0,\,0,\,0)\in\PP^5$ that is a double point.  Moreover, $\LL_2$ is a rational curve, which we now prove using $\V(\hat{f},\,\hat{g})$.

Let $i,\, a,\, b \in \kk$ be such that $i^2=-1,\,a^2=\a,\,b^2=2$. The localized curve $\V(\hat{f},\,\hat{g})$ is birationally equivalent to $\PP^1$, which can be seen by using the maps 
\[
\chi: \V(\hat{f},\,\hat{g}) \dashrightarrow \mathbb{A}^1 , \quad \text{defined by }\quad \chi(x,\,z)=\frac{-x(bz + i )}{a z},
\]
\noindent for $z\neq 0$, and
\[
\delta: \mathbb{A}^1\dashrightarrow \V(\hat{f},\,\hat{g}) , \quad \text{defined by }\quad\delta(t)=\left( \frac{a(1-t^2)}{2bt},\,\,\frac{1-t^2}{ib (1+t^2)} \right) ,
\]

\noindent for $t(t^2+1)\neq 0$. These maps are defined on almost all of their respective domains, and it is straightforward to check that $\chi\circ\delta=1=\delta\circ\chi$.

(III) Viewing $h=M_{12}^3 - M_{13}^2 M_{14} + M_{12} M_{14}^2$ as a polynomial in $\kk[M_{12},\,M_{13},\,M_{14}]$, the Jacobian matrix of $h$ is a $1\times 3 $ matrix that has rank one at all points of $\V(h)$. Thus, $\V(h)$ is nonsingular in the plane $\V\left(M_{23},\,M_{24},\,M_{34}\right)$, and so $\LL_3$ is a planar elliptic curve.

(IV) Analogously to (III), view $w=M_{13}^2 M_{23} - \a M_{23}^2 M_{34} + 2 M_{23} M_{34}^2 - \a M_{34}^3$ as a polynomial in $\kk[M_{13},\,M_{23},\,M_{34}]$, and consider its Jacobian matrix.  A computation shows that this $1\times 3$  matrix has rank one at all points of $\V(w)$.  Thus, $\V(w)$ is nonsingular in the plane $\V\left(M_{12},\,M_{14},\,M_{24}\right)$, and so $\LL_4$ is a planar elliptic curve.
\end{pf}

\begin{thm}\label{redthm}
For generic $\a$, the line scheme $\LL(\a)$ is a reduced scheme of degree twenty.
\end{thm}
\begin{pf}
Theorem \ref{8curves} shows that deg($\LL'(\a))=20$, and, since $\A(\a)$ satisfies the same conditions as those in \cite[Lemma 3.2]{ChV}, $\LL(\a)$ contains no embedded points. By \cite[Corollary 3.7]{CSV}, deg($\LL(\a))=20$ since dim($\LL(\a))=1$ and char($\kk)\neq 2$. Therefore, $\LL(\a)$ is given by $\LL'(\a)$ and hence is a reduced scheme.
\end{pf}

With the description of $\LL(\a)$ completed, we compute the points of intersection for the irreducible components of this scheme. We use the notation \begin{center} $E_i\; = \; \left(  \delta_{i1}    ,\,   \delta_{i2}    ,\,    \delta_{i3}   ,\,    \delta_{i4}   ,\,   \delta_{i5}   ,\,  \delta_{i6}  \right),$ \end{center} \noindent where $\delta_{ij}$ is the Kronecker-delta symbol.

\begin{cor}\label{intpoints}  The components of $\LL(\a)$ intersect at twenty
distinct points. The components that meet at exactly one point are given by{\rm :}
\\[-5mm]
\begin{center}
$\LL_2\cap\LL_3 \;=\; \LL_2\cap\LL_4  \;=\;  \LL_3 \cap\LL_4 \;=\; \lbrace E_2 \rbrace ,$ 
\\[2mm]
$\LL_{3}\cap\LL_{5a}\; = \; \lbrace E_3 \rbrace ,  \quad  \LL_{4}\cap\LL_{6a}\; = \;\lbrace  E_4 \rbrace ,   \,\,  $ and $\,\, \LL_{5a}\cap\LL_{6a}\;  = \; \lbrace E_5 
\rbrace${\rm ;}
\end{center}
\quad\\[-5mm]
\noindent and the components that meet at two distinct points are given by{\rm :}
\\[-2mm]
\begin{center}
\begin{tabular}{rclrcl}
$\LL_1\cap\LL_3$ & $=$ & $\lbrace ( 1,\,0,\,\ \pm i,\,0,\,0,\, 0  )\rbrace $,\hspace*{9mm} &
$\LL_2\cap\LL_{5b}$ & $=$& $\lbrace ( 0,\,0,\,0,\,0,\, a,\, \pm ib ) \rbrace $,
\\[2.5mm] 
$\LL_1\cap\LL_4$ & $=$& $\lbrace  ( 0,\,0,\,0 ,\, \a ,\,0,\, 1 \pm d ) \rbrace $, &
$\LL_2\cap\LL_{6b}$ &$=$& $ \lbrace (  b,\,0,\,0,\,0,\,\pm i,\, 0 )\rbrace $, \\[2.5mm]
$\LL_1\cap\LL_{5b}$ &$=$&  $\lbrace ( 0,\,0,\, \pm i a ,\,0,\,0,\, 1 ) \rbrace $,&
$\LL_{5a}\cap\LL_{5b}$ &$=$& $\lbrace( 0,\,0,\,b,\,0,\,\pm i,\, 0) \rbrace $, \\[2.5mm]
$\LL_1\cap\LL_{6b}$ &$=$&  $\lbrace ( \pm i a ,\,0,\,0 ,\, 1 ,\,0,\, 0 )\rbrace $, & $\LL_{6a}\cap\LL_{6b}$ &$=$& $\lbrace( 0,\,0,\,0,\,b,\,\pm ia,\, 0) \rbrace,$
\end{tabular}
\end{center}
\quad\\[-0mm]
\noindent where $i,\, a,\, b, \, d \in \kk$ and $i^2=-1,\,a^2=\a,\,b^2=2$, and $d^2=1-\a^2$.  Moreover, all other pairwise intersections are empty.
\end{cor} 

\begin{pf}
The result follows from direct computation.
\end{pf}

\begin{rmks}\label{midlrmks}

\begin{enumerate}[(a)]
\itemsep1mm
\item[]

\item In \cite{DT}, the multiplicity of these intersection points is computed using a localization argument; the point $E_2$, which is the unique singular point of $\LL_2$, lies on three distinct components and has multiplicity four, while all other intersection points have multiplicity two. 

\item In order to relate the intersection points of the components of $\LL(\a)$ to properties of $\A(\a)$, we now consider the intersection of certain right ideals in $\A(\a)$ and a family of normalizing sequences of $\A(\a)$. For any fixed $\delta,\,\epsilon \in \kk\setminus{\{0\}}$, the set
\[
\{ x_2^2, \ x_3^2,\  x_3 x_4 + \delta x_4 x_3 ,\  x_1 x_2 + \epsilon x_2 x_1 \}
\]
\noindent is a normalizing sequence of $A(\a)$. Let $J$ denote the ideal generated
by the elements of this normalizing sequence, and, given a point $p \in
\V(P)\subset \PP^5$ (i.e., a line in $\PP^3$), let $K_p$ denote its corresponding
right ideal in $A(\a)$. Motivated by \cite[\S 3.3B]{ChTh}, a calculation in
\cite{DT} shows that\\[-7mm]
\begin{center}
dim$_{\kk}\left( J_2 \cap K_p \right) = 2$ \quad for each point $p$ given by 
Corollary~\ref{intpoints},
\end{center}
\noindent where $J_2 $ denotes the span of the homogeneous degree-two elements in $J$.
For example, if $\delta  = -1$ and $\epsilon = 1$, the corresponding normalizing 
sequence is 
\[
\{ x_2^2,\ x_3^2,\  x_3 x_4 - x_4 x_3 , \ x_1 x_2 + x_2 x_1 \},
\]
\noindent and, to the two points $E_4 \pm i a E_1\in\LL_1 \cap \LL_{6b}$, we
associate the right ideals: \\[2mm]
\centerline{$\;x_4 A+\left(x_1\pm i ax_3\right)A=I_{\pm}$\,\!\! ;} \\[2mm]
a straightforward calculation shows that $I_{\pm}\cap J_2 = \kk \left( x_3 x_4 - x_4 x_3 \right) \oplus \kk x_2^2$. 
Moreover, based upon our calculations with various points of $\V(P)$ and any $\delta,\,\epsilon \in \kk\setminus{\{0\}}$, we conjecture that
\begin{center}
\begin{tabular}{rcl}

dim$_{\kk} \left( J_2 \cap K_p \right)  $ &  $= \; 1$ & if $ p$ lies on only one 
irreducible component of $\LL(\a)$, and \\[2mm]

dim$_{\kk} \left( J_2 \cap K_p \right) $ &  $ \geq \; 3 $ & if $ p \in \V(P) \setminus \LL(\a)$.

\end{tabular}

\end{center}

\end{enumerate}
\end{rmks}


\bigskip
\section{The Lines in $\PP^3$ Parametrized by $\LL(\a)$} \label{sec4}
In this section, we describe the lines in $\PP(V^\text{*})$ that are parametrized by the scheme $\LL(\a)$. After discussing how Pl\"ucker coordinates evaluate on a line in $\PP^3$, we give a description of the lines in $\PP^3$ that are parametrized by $\LL(\a)$ in Section \ref{p3lines}. In Section \ref{ptsonlines}, we discuss how many of those lines are incident to each point of the scheme $\p$.  Throughout, we identify $\PP(V^\text{*})$ with~$\PP^3$ since dim$(V)=4$, and we continue to use the notation $e_1,\ldots,e_4$ introduced in Section \ref{sec2}. 


\bigskip
\subsection{The Lines in $\PP^3$}\hfill\label{p3lines}
 
In this subsection, we describe the lines in $\PP^3$ that are parametrized by $\LL(\a)$.  For completeness, we first discuss how the Pl\"ucker coordinates relate to lines in $\PP^3$; more details may be found in \cite[\S 8.6]{CLO}.  Let $\ell$ be any line in $\PP^3$, and let $a=(a_1,\ldots,a_4)$ and $b=(b_1,\ldots,b_4)$ be two distinct points on $\ell$.  By representing $\ell$ as the $2\times 4$ matrix
\[
\begin{bmatrix}
a_1 & a_2 & a_3 & a_4 \\[2mm]
b_1 & b_2 & b_3 & b_4 
\end{bmatrix}
\]
\noindent of rank two, each point on $\ell$ may be realized, in homogeneous 
coordinates, as a linear combination of the rows of this matrix. Moreover,
infinitely many such matrices may be associated to any line $\ell$ and they are
all related by row operations.  The evaluation of the Pl\"ucker coordinate~$M_{ij}$ 
on the above matrix is defined to be the $2\times 2$ minor $a_ib_j-a_jb_i$. 
Notice that the Pl\"ucker polynomial $P=M_{12}M_{34}-M_{13}M_{24}+M_{14}M_{23}$ vanishes on this matrix, and hence $\V(P) \subset \PP^5 $ parametrizes all lines in $\PP^3$.

\bigskip

\noindent \textbf{(I)}
\ By Theorem \ref{8curves}, the component $\LL_1$, given by
\[
\LL_1 = \V\left(M_{13},\,M_{24},\,M_{12}M_{34}+M_{14}M_{23},\,M_{12}^2 + M_{14}^2 + \a M_{23}^2 - 2 M_{23} M_{34} +\a M_{34}^2\right),
\]
\noindent is a spatial elliptic curve in $\PP^5$.  Any line $\ell\subset
\PP(V^\text{*})$ 
corresponding to a point of $\LL_1$ can be represented by the $2\times 4$ matrix
\[
\begin{bmatrix}
a_1 & 0 & a_3 & 0 \\[2mm]
0 & b_2 & 0 & b_4 
\end{bmatrix}
\]
\quad\\ of rank two, where $a_j,\,b_j \in \kk$ for all $j$ such that 
$a_1^2 b_2^2 + a_1^2 b_4^2 + \a b_2^2 a_3^2 + 2 b_2 a_3^2b_4 + \a a_3^2 b_4^2=0$. If $p\in\ell$, then $p=(\mu_1 a_1,\,\mu_2 b_2,\,\mu_1 a_3,\,\mu_2 b_4)$ for some $\mu=(\mu_1,\,\mu_2)\in\PP^1$, so, for any $\mu\in\PP^1$, $p$ lies on the quartic surface
\[
\V\left(x_1^2 x_2^2 + x_1^2 x_4^2 + \a x_2^2 x_3^2 + 2 x_2 x_3^2x_4 + \a x_3^2 x_4^2\right)
\]
\noindent in $\PP^3$; thus $\ell$ lies on this surface.  Hence, $\ell$ corresponds to a point of $\LL_1$ if and only if 
\[
\ell = \V\left(  x_1 - \gamma x_3, \, x_2- \zeta x_4  \right),
\]
\noindent where $\gamma,\,\zeta \in \PP^1$ are such that $(\gamma^2+\a)(\zeta^2+1)+2\zeta=0$.

\bigskip \noindent \textbf{(II)}
	\ The component $\LL_2$ is the spatial rational curve given by 
\[
\LL_2  = \V\left(M_{14},\,M_{23},\,M_{12}M_{34}-M_{13}M_{24},\,M_{12}^2+2M_{24}^2+\a M_{34}^2\right).
\]
\noindent Any line $\ell \subset \PP(V^\text{*})$ 
corresponding to a point of $\LL_2$ can be represented by the rank-two $2\times 4$ matrix
\[
\begin{bmatrix}
a_1 & 0 & 0 & a_4 \\[2mm]
0 & b_2 & b_3 & 0 
\end{bmatrix},
\]  
\quad\\where $a_j,\,b_j \in \kk$ for all $j$ such that $a_1^2b_2^2+2b_2^2a_4^2+\a b_3^2a_4^2=0$.  Any point $p\in\ell$ is of the form $p=(\mu_1 a_1,\,\mu_2 b_2,\,\mu_2 b_3,\,\mu_1 a_4)$ for some $\mu=(\mu_1,\,\mu_2)\in\PP^1$, so, for any $\mu\in\PP^1$, $p$ lies on the quartic surface 
\[
\V\left(x_1^2x_2^2+2x_2^2x_4^2+\a x_3^2x_4^2\right)
\]
in $\PP^3$; hence, $\ell$ lies on this quartic surface.  Thus, $\ell$ corresponds to a point of $\LL_2$ if and only if
\[
\ell=\V\left(x_1-\beta x_4,\, x_3-\omega x_2\right),
\]
\noindent where $\beta,\,\omega\in\PP^1$ are such that $\a\omega^2 + \beta^2+2 = 0$.

\bigskip \noindent \textbf{(III)}
	\ The component $\LL_3$ is given by
	\[
\V\left(M_{23},\,M_{24},\,M_{34},\, M_{12}^3 - M_{13}^2 M_{14} + M_{12} M_{14}^2 \right),
	\]
which is a planar elliptic curve.  Any line in~$\PP(V^\text{*})$ 
parametrized by~$\LL_3$ can be represented by the rank-two matrix
\[
\begin{bmatrix}
0 & a_2 & a_3 & a_4 \\[2mm]
1 & 0 & 0 & 0 
\end{bmatrix},
\]
\quad\\[-1mm]
where $a_j\in \kk$ for all $j$ such that $a_2^3-a_3^2a_4+a_2a_4^2=0$. Thus, $\LL_3$ parametrizes all lines 
in~$\PP(V^\text{*})$ 
that pass through~$e_1$ and meet the planar curve 
$\V\left(x_1,\,x_2^3-x_3^2x_4+x_2x_4^2\right)$.  Note that this planar curve is nonsingular (since char$(\kk)= 0$) and hence is an elliptic curve.

\bigskip \noindent \textbf{(IV)}
	\ The planar elliptic curve 
	\[
	 \LL_4 = \V\left(M_{12},\,M_{14},\,M_{24},\,M_{13}^2 M_{23} - \a M_{23}^2 M_{34} + 2 M_{23} M_{34}^2 - \a M_{34}^3\right)
	\]
\noindent parametrizes lines 
in~$\PP(V^\text{*})$ 
that can be represented by the rank-two matrix
\[
\begin{bmatrix}
a_1 & a_2 & 0 & a_4 \\[2mm]
0 & 0 & 1 & 0 
\end{bmatrix},
\]
\noindent where  $a_j\in \kk$ for all $j$ such that $ a_1^2a_2+\a a_2^2a_4+2a_2a_4^2+\a a_4^3=0$; these lines are those passing through~$e_3$ that meet the planar curve $\V\left(x_3,\,x_1^2x_2+\a x_2^2x_4+2x_2x_4^2+\a x_4^3\right)$.  This planar curve is nonsingular since $\a$ is generic (and char$(\kk)= 0$) and so is an elliptic curve.

\bigskip \noindent \textbf{(V)}
	\ The lines parametrized by $\LL_5$ consist of two distinct components:
\[
\LL_{5a} = \V\left(M_{12},\,M_{13},\,M_{23},\,M_{34}\right) \quad \text{and}\quad \LL_{5b} =\V\left(M_{12},\,M_{13},\,M_{23},\,M_{14}^2+2M_{24}^2+\a M_{34}^2\right),
\] 
\noindent where $\LL_{5a}$ is a line and $\LL_{5b}$ is a nonsingular conic.  Any line 
in~$\PP(V^\text{*})$ 
corresponding to a point of $\LL_{5a}$ can be represented by the rank-two matrix 
\[
\begin{bmatrix}
a_1 & a_2 & 0 & 0 \\[2mm]
0 & 0 & 0 & 1
\end{bmatrix},
\] 
\noindent which describes all lines lying on $\V(x_3)$ that contain $e_4$.  The lines given by $\LL_{5b}$ are represented by the rank-two matrix 
\[
\begin{bmatrix}
a_1 & a_2 & a_3 & 0 \\[2mm]
0 & 0 & 0 & 1
\end{bmatrix},
\] 
\quad\\[-1mm]
where $a_j\in \kk$ for all $j$ such that $a_1^2+2a_2^2+\a a_3^2=0$.  It follows that $\LL_{5b} $ parametrizes all lines 
in~$\PP(V^\text{*})$ 
that lie on the singular rank-three quadric $\V(x_1^2+2x_2^2+\a x_3^2)$ and pass through its unique singular point,~$e_4$.

\bigskip \noindent \textbf{(VI)}
	\ Similar to Case (V), $\LL_6$ consists of the two components
\[
\LL_{6a} = \V\left(M_{12},\,M_{13},\,M_{14},\,M_{34}\right) \quad \text{and}\quad \LL_{6b} =\V\left(M_{13},\,M_{14},\,M_{34},\,M_{12}^2+\a M_{23}^2+2 M_{24}^2\right),
\]
\noindent where $\LL_{6a}$ is a line and $\LL_{6b}$ is a nonsingular conic. By comparison with $\LL_{5a}$, we find that $\LL_{6a}$ parametrizes all lines on the plane $\V(x_1)$ that contain $e_2$. Similarly, our description of $\LL_{5b}$ indicates that $\LL_{6b}$ parametrizes all lines 
in~$\PP(V^\text{*})$ 
that lie on the singular rank-three quadric $\V(x_1^2+\a x_3^2 +2 x_4^2)$ and pass through its unique singular point, $e_2 $.


\bigskip

\subsection{The Lines of the Line Scheme That Contain Points of the Point Scheme}\hfill\label{ptsonlines}

In this subsection, we compute how many lines in $\PP(V^{\text{*}})$ parametrized by~$\LL(\a)$ contain a given point of the scheme $\p$. By \cite[Remark~3.2]{SV1}, if the number of such lines is finite,
then it is six (counting multiplicity); hence, the generic case is considered to be six distinct lines. Recall from Proposition \ref{ptS} that $\p=\bigcup _{i=0}^4 \Z_i $ and $\Z_0 = \lbrace e_1,\ldots,e_4\rbrace $. 

\begin{thm}\label{6lines}
Let $\a$ be generic.
\begin{enumerate}
\itemsep1mm
\item[\textup{(a)}]
     For any $j \in \{ 1, \ldots, 4\}$, $e_j$ lies on infinitely many lines in 
     $\PP(V^{\text{\rm *}})$ 
     that are parametrized by $\LL(\a)$.
\item[\textup{(b)}]
     Each point of $\Z_1 \cup \Z_2 $ lies on exactly six distinct lines of
     those parametrized by $\LL(\a)$.
\item[\textup{(c)}] 
	Each point of $\Z_3 \cup \Z_4 $ lies on four distinct lines of those parametrized by $\LL(\a)$; two of which have multiplicity two and two of which have multiplicity one.

\end{enumerate}
\end{thm} 

\begin{pf}
As (a) follows from (III) -- (VI) in Section~4.1, we discuss  (b) and (c).  

Let $p_1=(\lambda_1,\,1,\,\lambda_3,\,1)\in\Z_1$, where $\lambda_1^2=-2(1+\a) $ and $ \lambda_3^2=2\,$.  Setting $\gamma= \lambda_1/\lambda_3$ and $\zeta=1$, we have $(\gamma^2+\a)(\zeta^2+1)+2\zeta=0$, and thus $\V\left(x_1 - \gamma x_3,\,x_2-\zeta x_4\right)=\ell_{11}$ contains $p_1$ and corresponds to a point of $\LL_1$. Clearly, no other line given by $\LL_1 $ contains $p_1$.  Similarly, setting $\beta=\lambda_1$ and $\omega=\lambda_3$ implies that $\a\omega^2=-(\beta^2+2)$, and hence $p_1\in\V\left(x_1-\beta x_4,\,x_3- \omega x_2\right)=\ell_{12}$, which is a line corresponding to an element of $\LL_2$.  Clearly, no other line given by $\LL_2$ contains~$p_1$. 

Let $r_1 = (0,\, 1,\, \lambda_3,\, 1)$, and let $\ell_{13}$ denote the
line through $e_1$ and $r_1$.  Since $\lambda_3^2=2$, we have $r_1 \in
\V\left(x_1,\,x_2^3-x_3^2x_4+x_2x_4^2\right)$; hence, $\ell_{13}$ corresponds to an element of $\LL_3$, and $p_1\in\ell_{13}$.  Conversely,
let $r_1' = (0,\, b_2,\, b_3,\, b_4) \in \V\left(x_1,\,x_2^3-x_3^2x_4+x_2x_4^2\right)$. If $p_1$ lies on the line through $r_1'$ and $e_1$, then there exists $(\mu_1, \, \mu_2) \in \PP^1$ such that $p_1 = 
(\mu_1, \, \mu_2 b_2 ,\, \mu_2 b_3,\, \mu_2 b_4)$. Thus,
$\mu_2 b_2 = \mu_2 b_4\neq 0$ and $\lambda_3 = b_3/b_2$.
Hence, $r_1' = 
(0,\, 1,\, \lambda_3,\, 1 ) = r_1$, and so $\,\ell_{13}\,$ is the only line parametrized by $\LL_3$ that contains $p_1$.

Let $r_3 = (\lambda_1,\, 1,\, 0,\, 1)$, and let $\ell_{14}$ denote the
line through $e_3$ and $r_3$.  The fact that
$\lambda_1^2+2(1+\a)=0$ implies that $r_3 \in
\V\left(x_3,\,x_1^2x_2+\a x_2^2x_4+2x_2x_4^2+\a x_4^3\right)$; thus, $\ell_{14}$ corresponds to an element of $\LL_4$, and $p_1\in\ell_{14}$.  An argument similar to that of $\LL_3$ proves that no other line given by $\LL_4$ 
contains~$p_1$.  

As $p_1\notin\V(x_3)$, no line given by $\LL_{5a} $ contains $p_1$. For $\LL_{5b} $, let $r_4= (\lambda_1,\, 1,\, \lambda_3,\, 0)$, and let $\ell_{15}$ denote the line through $e_4$ and $r_4$, so $p_1\in\ell_{15}$.  Noting again that $\lambda_1^2+2(1+\a)=0$, we have $r_4\in \V\left(x_1^2+2x_2^2+\a x_3^2\right)$, which implies that $\ell_{15}$ corresponds to an element of $\LL_{5b}$. An argument similar to that of $\LL_3$ proves that no other line given by $\LL_{5b}$ contains~$p_1$.

As $p_1\notin\V(x_1)$, no lines given by $\LL_{6a} $ contain $p_1$. For $\LL_{6b} $, let $\ell_{16}$ denote the line through $e_2$ and $r_2$, where $r_2= (\lambda_1,\, 0,\, \lambda_3,\, 1)$, so $p_1\in\ell_{16} $. The conditions on the $\lambda_i$ imply that $r_2\in \V\left(x_1^2+\a x_3^2 +2 x_4^2\right)$; thus, $\ell_{16}$ corresponds to an element of $\LL_{6b}$. An argument similar to that of $\LL_3$ proves that no other line given by $\LL_{6b}$ contains~$p_1$.  Hence, there are exactly six distinct lines given by $\LL(\a)$ that contain $p_1$.


In order to consider points of $\Z_2$, we first consider the algebra isomorphism $\Phi$ from $\A(-\a)$ to $\A(\a)$ defined by $\Phi(x_1)=x_1$, $\Phi(x_4)=x_4$, $\Phi(x_2)=-x_2$, $\Phi(x_3)=i x_3$. Notice that $\Z_2=\Phi(\widehat{\Z_1})$, where $\widehat{\Z_1} = \{ \,(\lambda_1,\,1,\,\lambda_3,\,1)\;| \; \lambda_1^2=-2(1 -\a) \, , \; \lambda_3^2=2\, \}, $ and so corresponds to four point modules of $\A(-\a)$, by Proposition \ref{ptS}. Under the action of $\Phi$, the six distinct lines parametrized by $\LL(-\a)$ containing a given point, $\widehat{p_1}$, of $\widehat{\Z_1}$ map to six distinct lines parametrized by $\LL(\a)$ that contain the point $\Phi(\widehat{p_1})\in\Z_2$. Hence, (b) follows.


Let $p_3=(\lambda_1,\,\lambda_2,\,0,\,1)\in\Z_3$, where $\lambda_1^2=-2$ and $\lambda_2^2=-1$.  By setting $\zeta=\lambda_2$, we have $p_3\in\V\left(x_3,\,x_2-\zeta x_4\right)=\ell_{31}$, which corresponds to a point of $\LL_1$ since $\lambda_2^2+1=0 $. Clearly, any other line corresponding to a point of $\LL_1 $ does not contain $p_3$. In fact, $\ell_{31}$ corresponds to a point of $\LL_1\cap\LL_3$ since it is the line that passes through $e_1$ and $t_1$, where $t_1=(0,\lambda_2,\,0,\,1)\in\V\left(x_1,\,x_2^3-x_3^2x_4+x_2x_4^2\right)$.  If $p_3$ lies on the line through $e_1$ and $t_1'$ for any $t_1'=(0,\,b_2,\,b_3,\,b_4)\in\V\left(x_1,\,x_2^3-x_3^2x_4+x_2x_4^2\right)$, then there exists $(\mu_1,\,\mu_2)\in\PP^1$ such that $p_3=(\mu_1,\,\mu_2 b_2,\,\mu_2 b_3,\,\mu_2 b_4)$.  As $\mu_2\neq 0$, it follows that $b_3=0\neq b_4$ and $\lambda_2=b_2/b_4$.  Thus, $t_1'=(0,\,\lambda_2,\,0,\,1)=t_1$, and so $\ell_{31}$ is the only line corresponding to a point of $\LL_3$ that contains $p_3$. It follows that $\ell_{31}$ corresponds to a point of both $\LL_1$ and $\LL_3$.


By setting $\beta=\lambda_1$ and $\omega=0$, we have $p_3\in \V\left(x_1-\lambda_1 x_4,\,x_3\right)=\V\left(x_1-\beta x_4,\,x_3-\omega x_2\right)=\ell_{32}$, which corresponds to a point of $\LL_2$ since $\a\omega^2=-(\beta^2+2)$.  Clearly, no other line given by $\LL_2$ contains $p_3$.  In fact, $\ell_{32}$ corresponds to a point of $\LL_2\cap\LL_{6b}$ since it is the line that passes through $e_2$ and $t_2$, where $t_2=(\lambda_1,\,0,\,0,\,1)\in\V\left(x_1^2+\a x_3^2 + 2x_4^2\right)$. An argument similar to that of $\LL_3$ shows that no other line given by $\LL_{6b}$ contains~$p_3$.


Let $\ell_{34}$ denote the line through $e_3$ and $p_3$. In particular, $\ell_{34}$ 
corresponds to an element of $\LL_4$ since $p_3\in\V\left(x_3,\,x_1^2x_2+\a x_2^2x_4+2x_2x_4^2+\a x_4^3\right)$.  Let $t_3 = (b_1,\, b_2,\, 0,\, b_4) \in \V\big(x_3,\,x_1^2x_2+\a x_2^2x_4+2x_2x_4^2+\a x_4^3\big)$. If $p_3$ lies on the line through $t_3$ and $e_3$, then there exists $(\mu_1, \, \mu_2) \in \PP^1$ such that $p_3 = 
(\mu_2 b_1\, \mu_2 b_2 ,\, \mu_1,\, \mu_2 b_4)$. Thus,
$\mu_2 b_4 \neq 0$, $\lambda_1 = b_1/b_4$, and $\lambda_2 = b_2/b_4$. Hence $t_3= 
(\lambda_1,\, \lambda_2,\, 0,\, 1) = p_3$, and so $\ell_{34}$ is the only line of those parametrized by $\LL_4$ that contains~$p_3$.

Let $\ell_{35}$ denote the line connecting $p_3$ and $e_4$; in particular, $p_3\in\ell_{35}\subset \V(x_3)$, and $\ell_{35}$ is the only line given by $\LL_{5a} $ that contains $p_3$.  There are no lines parametrized by $\LL_{5b} $ containing $p_3$ since $p_3 \notin \V(x_1^2 + 2 x_2^2 + \a x_3^2)$, and additionally, no line given by $\LL_{6a} $ contains $p_3$ as $p_3\notin\V(x_1)$.   Thus, there are four distinct lines, with a total multiplicity of six (by \cite[Remark 3.2]{SV1}), parametrized by $\LL(\a)$ that contain $p_3$.


Let $p_4=(0,\,\lambda_2,\,\lambda_3,\,1)\in\Z_4$, where $\a\lambda_2^2+2\lambda_2+\a=0$ and $\a\lambda_3^2=-2\lambda_2^2\,$.  Letting $\gamma=0$ and $\zeta=\lambda_2$, so that $(\gamma^2+\a)(\zeta^2+1)+2\zeta=0$, reveals that the line $\V\left(x_1-\gamma x_3,\, x_2-\zeta x_4\right)= \V\left(x_1,\,x_2-\lambda_2x_4\right)=\ell_{41}$ contains $p_4$ and is among those parametrized by $\LL_1$; moreover, any other line given by $\LL_1 $ does not contain $p_4$. In fact, $ \ell_{41}$ corresponds to a point of $\;\LL_1\cap \LL_4\;$ since it is the line that passes through $e_3$ and $q_3$, where $q_3=(0,\,\lambda_2,\,0,\,1)\in\V\big(x_3,\,x_1^2x_2+\a x_2^2x_4+ 2x_2x_4^2+\a x_4^3\big)$.  Conversely, let $q_3'=(b_1,\,b_2,\,0,\,b_4)\in\V\big(x_3,\,x_1^2x_2+\a x_2^2x_4+ 2x_2x_4^2+\a x_4^3\big)$.  If $p_4$ lies on the line through $e_3$ and $q_3'$, then there exists $(\mu_1,\,\mu_2)\in\PP^1$ such that $p_4=(\mu_2 b_1,\,\mu_2 b_2,\,\mu_1,\,\mu_2 b_4)$.  As $\mu_2\neq 0$, it follows that $b_1= 0 \neq b_4$ and $\lambda_2=b_2/b_4$.  Thus, $q_3'=(0,\,\lambda_2,\,0,\,1)=q_3$, and so $\ell_{41}$ is the only line corresponding to a point of $\LL_4$ that contains $p_4$.


By setting $\beta=0$ and $\omega= \lambda_3 / \lambda_2$, we have $\a\omega^2=-(\beta^2+2)$, and thus the line $\V (x_1-\beta x_4,\,x_3-\omega x_2 )=\V\left(x_1,\, \lambda_3 x_2 - \lambda_2 x_3 \right) = \ell_{42}$ contains $p_4$ and also corresponds to a point of $\LL_2$. Clearly, no other line given by $\LL_2$ contains $p_4$.  Moreover, $\ell_{42}$ corresponds to a point of $\LL_2\cap \LL_{5b} $ since it is the line that passes through $e_4$ and $q_4$, where $q_4=(0,\,\lambda_2,\,\lambda_3,\,0)\in\V(x_1^2+2 x_2^2 + \a x_3^2\,)$.  An argument similar to that for $\LL_4$ shows that no other line given by $\LL_{5b}$ contains $p_4$.


Let $\ell_{43}$ denote the line through $e_1$ and $p_4$.  Since $p_4 \in
\V\left(x_1,\,x_2^3 - x_3^2x_4 + x_2 x_4^2\right)$, $\ell_{43}$ 
corresponds to an element of $\LL_3$.  Conversely,
let $q_1 = (0,\, b_2,\, b_3,\, b_4) \in \V\left(x_1,\,x_2^3 - x_3^2x_4 + x_2 x_4^2\right)$. If $p_4$ lies on the line through $q_1$ and $e_1$, then there exists $(\mu_1, \, \mu_2) \in \PP^1$ such that $p_4 = 
(\mu_1,\, \mu_2 b_2 ,\, \mu_2 b_3,\, \mu_2 b_4)$. Thus, $\mu_2 b_4 \neq 0$, $\lambda_2 = b_2/b_4$, and $\lambda_3 = b_3/b_4$.  Hence, $q_1 =(0,\,\lambda_2,\,\lambda_3,\,1) = p_4$, and so $\ell_{43}$ is the only line of those parametrized by $\LL_3$ that contains $p_4$.

Let $\ell_{46}$ denote the line connecting $p_4$ and $e_2$; in particular, $p_4\in\ell_{46}\subset \V(x_1)$, and $\ell_{46}$ is a line given by $\LL_{6a} $. Clearly no other line given by $\LL_{6a} $ contains $p_4$.

As $p_4\notin \V(x_3)$, there are no lines parametrized by $\LL_{5a} $ that contain $p_4$.  Additionally, no lines given by $\LL_{6b} $ contain $p_4$ because $p_4\notin \V(x_1^2 + \a x_3^2 + 2 x_4^2)$ (since $\a$ is generic).   Thus, there are four distinct lines parametrized by $\LL(\a)$ that contain $p_4$, with a total multiplicity of six (by \cite[Remark 3.2]{SV1}).
\end{pf}

\begin{rmks}\label{finalrmks}

\begin{enumerate}[(a)]
\itemsep1mm\item[]
\item In the proof of Theorem \ref{6lines}(b), a certain isomorphism $\Phi: \A(-\a)\rightarrow \A(\a)$ was used in order to deduce the result for the points of $\Z_2$ from that of the points of $\Z_1$. Although a similar symmetry between the points of $\Z_3$ and $\Z_4$ is likely, it is not clear, however, if an analogous map (or twisting system) exists in that context that would simplify the proof of Theorem \ref{6lines}(c).

\item In \cite{ChV}, Conjecture 4.2 proposes that the line scheme of a generic quadratic quantum $\PP^3$ should be the union of two spatial elliptic curves and four planar elliptic curves. The results of this paper support that conjecture in the following sense. If we write $\LL_5 = \LL_{5a}\cup\LL_{5b}$ and $\LL_6 = \LL_{6a}\cup\LL_{6b}$, then $ \LL(\a)=\bigcup_{i=1}^6 \LL_i$, where $\LL_5$ and $\LL_6$ may be viewed as degenerations of planar elliptic curves and $\LL_2$ may be viewed as a degeneration of a spatial elliptic curve. Moreover, Theorem \ref{6lines} implies that each point $p\in\p\setminus \Z_0$ lies on a line $\ell_i$ in $\PP(V^*)$ that is represented by a point on $\LL_i,$ for each $ i=1,\ldots,6$. In particular, if \cite[Conj. 4.2]{ChV} is correct, then $\A(\a)$ is not a generic quadratic quantum $\PP^3$.

\item The algebras in \cite{EE} have point schemes consisting of twenty distinct points and line schemes consisting of three spatial elliptic curves and four conics. So either \cite[Conj. 4.2]{ChV} needs to be modified (possibly applying only to certain classes of algebras, such as regular graded skew Clifford algebras), or the algebras in \cite{EE} are not generic quadratic quantum $\PP^3$s, or perhaps there exist many classes of generic quadratic quantum $\PP^3$s.

\end{enumerate}

\end{rmks}

\bigskip
 
\section{Appendix}\label{app}

In this section, we list the polynomials that define $\p(\a)$ and $\LL(\a)$ for generic $\a\in\kk$.

\medskip

\subsection{Polynomials Defining the Point Scheme}\label{app1}\hfill

The following are the polynomials given by the fifteen $4\times 4$ minors of the matrix $M$ given in Section \ref{sec2}.
\begin{enumerate}[\qquad\thesubsection.1.\quad]
\itemsep1mm
\item\label{pp1} $x_1 x_3(x_1^2+2x_2^2+\a x_3^2) $,

\item\label{pp2} $x_1 x_3(x_3^2 - 2x_2x_4)$,

\item\label{pp3} $x_1x_3(x_2-x_4)(x_2+x_4)$,

\item\label{pp4} $x_2x_3(x_1^2+2x_2^2+\a x_3^2)$,

\item\label{pp5} $x_3x_4(x_1^2+2x_2^2+\a x_3^2)$,

\item\label{pp6} $x_3(x_2^3-x_3^2x_4+x_2x_4^2)$,

\item\label{pp7} $x_1(2x_2x_3^2+x_1^2x_4-2x_2^2x_4+\a x_3^2x_4)$,

\item\label{pp8} $x_1x_2(x_1^2+\a x_3^2+2x_4^2)$,

\item\label{pp9} $x_1(-x_2x_3^2+x_2^2x_4+x_4^3)$,

\item\label{pp10} $ x_1^2x_2^2+\a x_2^2x_3^2 +2x_2x_3^2x_4 +x_1^2x_4^2 + \a x_3^2x_4^2$,

\item\label{pp11} $x_1^2x_3^2-x_1^2x_2x_4+2x_2^3x_4+\a x_2x_3^2 x_4$,

\item\label{pp12} $ x_1^2x_2^2+2x_2^2x_4^2+\a x_3^2x_4^2 $,

\item\label{pp13} $ x_4(x_2^3 - x_3^2x_4 + x_2x_4^2)$,

\item\label{pp14} $x_1(x_2^3-x_3^2x_4+x_2x_4^2)$,

\item\label{pp15} $x_3(x_1^2x_2+\a x_2^2x_4 + 2x_2x_4^2+\a x_4^3)$.
\end{enumerate}

\subsection{Polynomials Defining the Line Scheme}\label{app2}\hfill

The following are the 46~polynomials in the $M_{ij}$ coordinates from 
Section~\ref{sec3}.1 that define the line scheme $\LL(\a)$ of $\A(\a)$:
\\[-3mm]

\medskip

\begin{enumerate}[\qquad\thesubsection.1.\quad]
\itemsep1mm
\setcounter{enumi}{-1}
\item\label{lp0}
$P = M_{12} M_{34} - M_{13} M_{24} + M_{14} M_{23}$,

\item\label{lp1}
$2M_{13} M_{14} M_{23} M_{24}$,

\item\label{lp2}
$M_{12} (\a M_{12} M_{13} M_{23}+2 M_{13} M_{14} M_{23}-2 M_{12} M_{14} M_{24}-\a M_{13} M_{14} M_{34})$,

\item\label{lp3}
$M_{13} (\a M_{12} M_{13} M_{23}+2 M_{13} M_{14} M_{23}-2 M_{12} M_{14} M_{24}-\a M_{13} M_{14} M_{34})$,

\item\label{lp4}
$M_{14} (\a M_{12} M_{13} M_{23}+2 M_{13} M_{14} M_{23}-2 M_{12} M_{14} M_{24}-\a M_{13} M_{14} M_{34})$,

\item\label{lp5}
$M_{23} (\a M_{12} M_{13} M_{23}+2 M_{13} M_{14} M_{23}-2 M_{12} M_{14} M_{24}-\a M_{13} M_{14} M_{34})$,

\item\label{lp6}
$M_{24} (\a M_{12} M_{13} M_{23}+2 M_{13} M_{14} M_{23}-2 M_{12} M_{14} M_{24}-\a M_{13} M_{14} M_{34}) $,

\item\label{lp7}
$M_{34} (\a M_{12} M_{13} M_{23}+2 M_{13} M_{14} M_{23}-2 M_{12} M_{14} M_{24}-\a M_{13} M_{14} M_{34}) $,

\item\label{lp8}
$M_{23} (\a M_{12} M_{13} M_{23}+2 M_{13} M_{14} M_{23}+2 M_{12} M_{14} M_{24}-\a M_{13} M_{14} M_{34})$,

\item\label{lp9}
$M_{24} (\a M_{12} M_{13} M_{23}-2 M_{13} M_{14} M_{23}+2 M_{12} M_{14} M_{24}+\a M_{13} M_{14} M_{34})$,

\item\label{lp10}
$M_{14} (M_{12} M_{13} M_{23}-M_{13} M_{14} M_{34}+2 M_{23} M_{24} M_{34})$,

\item\label{lp11}
$M_{24} (M_{12} M_{13} M_{23}-M_{13} M_{14} M_{34}+2 M_{23} M_{24} M_{34})$,

\item\label{lp12}
$M_{12} (M_{13} M_{14} M_{23}+M_{12} M_{14} M_{24}-\a M_{23} M_{24} M_{34})$,

\item\label{lp13}
$M_{13} (M_{13} M_{14} M_{23}+M_{12} M_{14} M_{24}-\a M_{23} M_{24} M_{34})$,

\item\label{lp14}
$M_{23} (M_{13} M_{14} M_{23}+M_{12} M_{14} M_{24}-\a M_{23} M_{24} M_{34})$,

\item\label{lp15}
$M_{12} (M_{13} M_{14} M_{23}-M_{12} M_{14} M_{24}+\a M_{23} M_{24} M_{34})$,

\item\label{lp16}
$M_{13} (M_{13} M_{14} M_{23}-M_{12} M_{14} M_{24}+\a M_{23} M_{24} M_{34})$,

\item\label{lp17}
$M_{14} (M_{13} M_{14} M_{23}-M_{12} M_{14} M_{24}+\a M_{23} M_{24} M_{34})$,

\item\label{lp18}
$M_{23} (M_{13} M_{14} M_{23}-M_{12} M_{14} M_{24}+\a M_{23} M_{24} M_{34})$,

\item\label{lp19}
$M_{24} (M_{13} M_{14} M_{23}-M_{12} M_{14} M_{24}+\a M_{23} M_{24} M_{34})$,

\item\label{lp20}
$M_{34} (M_{13} M_{14} M_{23}-M_{12} M_{14} M_{24}+\a M_{23} M_{24} M_{34})$,

\item\label{lp21}
$M_{12} (M_{12} M_{13} M_{23} - M_{13} M_{14} M_{34} + 2 M_{23} M_{24} M_{34})$,

\item\label{lp22}
$
M_{13} (M_{13} M_{14} M_{23} - M_{12} M_{14} M_{24} - \a M_{23} M_{24} M_{34})$,

\item\label{lp23}
$M_{12} M_{13} M_{14} M_{23} - M_{12}^2 M_{14} M_{24} - \a M_{14} M_{23}^2 M_{24} - \a M_{13} M_{23} M_{24}^2$,

\item\label{lp24}
$
M_{12} M_{13} M_{23}^2 + M_{12} M_{14} M_{24} M_{34} + 2 M_{23}^2 M_{24} M_{34} - \a M_{23} M_{24} M_{34}^2$,

\item\label{lp25}
$M_{12} M_{13} M_{24}^2 - M_{12} M_{13} M_{23} M_{34} + M_{14}^2 M_{24} M_{34} + 2 M_{24}^3 M_{34} - 2 M_{23} M_{24} M_{34}^2 + \a M_{24} M_{34}^3
$,

\item\label{lp26}
$M_{12} M_{13}^2 M_{23} - M_{12}^2 M_{13} M_{24} + 2 M_{13} M_{23} M_{24} M_{34} - 2 M_{12} M_{24}^2 M_{34} - \a M_{13} M_{24} M_{34}^2$,

\item\label{lp27}
$M_{12} M_{14}^2 M_{23} - M_{14}^3 M_{34} - 2 M_{14} M_{24}^2 M_{34} + \a M_{12} M_{23} M_{34}^2 + 2 M_{14} M_{23} M_{34}^2 - \a M_{14} M_{34}^3$,

\item\label{lp28}
$M_{12} M_{13}^2 M_{23} - M_{12}^2 M_{13} M_{24} - 2 \a M_{13} M_{23}^2 M_{24} + 2 M_{13} M_{23} M_{24} M_{34} - 2 M_{12} M_{24}^2 M_{34} - \a M_{13} M_{24} M_{34}^2$,

\item\label{lp29}
$M_{12} M_{13}^2 M_{23} + M_{12}^2 M_{14} M_{23} + M_{13}^2 M_{14} M_{34} - M_{12} M_{14}^2 M_{34} - \a M_{12} M_{23}^2 M_{34} + 2 M_{12} M_{23} M_{34}^2 + \a M_{14} M_{23} M_{34}^2$,

\item\label{lp30}
$M_{12}^2 M_{13} M_{23} + M_{13}^2 M_{14} M_{24} - M_{12} M_{14}^2 M_{24} + 2 M_{12} M_{23} M_{24} M_{34} + \a M_{14} M_{23} M_{24} M_{34}$,

\item\label{lp31}
$M_{12}^2 M_{13} M_{23}-M_{12}^3 M_{24}-\a M_{12} M_{23}^2 M_{24}+2 M_{13} M_{23} M_{24}^2-2 M_{12} M_{24}^3-\a M_{13} M_{24}^2 M_{34}$,

\item\label{lp32}
$M_{12}^2 M_{14} M_{23} - M_{12} M_{14}^2 M_{34} - \a M_{12} M_{23}^2 M_{34} - 2 M_{14} M_{23}^2 M_{34} + \a M_{14} M_{23} M_{34}^2$,

\item\label{lp33}
$M_{12}^2 M_{14} M_{23}+4 M_{14} M_{23} M_{24}^2-M_{12} M_{14}^2 M_{34}-\a M_{12} M_{23}^2 M_{34}-2 M_{14} M_{23}^2 M_{34}+\a M_{14} M_{23} M_{34}^2$,

\item\label{lp34}
$
M_{12}^2 M_{13} M_{24} + 2 M_{13} M_{14}^2 M_{24} - M_{13}^2 M_{14} M_{34} + 2 M_{12} M_{24}^2 M_{34} + \a M_{13} M_{24} M_{34}^2$,

\item\label{lp35}
$\a M_{12}^2 M_{13} M_{23} + 2 M_{12} M_{13} M_{14} M_{23} + \a M_{13} M_{14}^2 M_{23} + 2 M_{12}^2 M_{14} M_{24} + \a M_{13}^2 M_{14} M_{24}$,

\item\label{lp36}
$M_{12}^3 M_{14} - M_{13}^2 M_{14}^2 + M_{12} M_{14}^3 - \a M_{12}^2 M_{23} M_{34} - 2 M_{12} M_{14} M_{23} M_{34} - \a M_{14}^2 M_{23} M_{34}$,

\item\label{lp37}
$M_{12}^3 M_{23} + \a M_{12} M_{23}^3 + 2 M_{14} M_{23}^3 + 2 M_{12} M_{23} M_{24}^2 - M_{12}^2 M_{14} M_{34} - \a M_{14} M_{23}^2 M_{34}$,

\item\label{lp38}
$M_{12}^3 M_{13} - M_{13}^3 M_{14} + M_{12} M_{13} M_{14}^2 + 2 M_{12}^2 M_{24} M_{34} + \a M_{13}^2 M_{24} M_{34}$,

\item\label{lp39}
$M_{12}^4 - M_{12} M_{13}^2 M_{14} + M_{12}^2 M_{14}^2 + \a M_{12}^2 M_{23}^2 + 2 M_{12} M_{14} M_{23}^2 + \a M_{14}^2 M_{23}^2 + 2 M_{12}^2 M_{24}^2 + \a M_{13}^2 M_{24}^2$,

\item\label{lp40}
$\a M_{13} M_{14} M_{23}^2+\a M_{13}^2 M_{23} M_{24}-2 M_{13} M_{14} M_{23} M_{34}+2 M_{12} M_{14} M_{24} M_{34}+\a M_{13} M_{14} M_{34}^2
$,

\item\label{lp41}
$M_{13}^2 M_{23}^2+M_{12} M_{14} M_{23}^2-\a M_{23}^3 M_{34}+M_{12} M_{14} M_{34}^2+2 M_{23}^2 M_{34}^2-\a M_{23} M_{34}^3
$,

\item\label{lp42}
$M_{13}^3 M_{23}-M_{12} M_{13}^2 M_{24}-\a M_{13} M_{23}^2 M_{34}+2 M_{13} M_{23} M_{34}^2-2 M_{12} M_{24} M_{34}^2-\a M_{13} M_{34}^3
$,

\item\label{lp43}
$2 M_{14}^2 M_{23} M_{24}+2 M_{13} M_{14} M_{24}^2-\a M_{12} M_{13} M_{23} M_{34}-2 M_{13} M_{14} M_{23} M_{34}+\a M_{13} M_{14} M_{34}^2
$,

\item\label{lp44}
$M_{14}^2 M_{23} M_{24} + M_{13} M_{14} M_{24}^2 - M_{13} M_{14} M_{23} M_{34} + \a M_{23} M_{24} M_{34}^2
$,

\item\label{lp45}
$M_{14}^2 M_{23}^2+M_{13}^2 M_{24}^2-M_{13}^2 M_{23} M_{34}+M_{14}^2 M_{34}^2+\a M_{23}^2 M_{34}^2+2 M_{24}^2 M_{34}^2-2 M_{23} M_{34}^3+\a M_{34}^4.
$

\end{enumerate}


\bigskip

\begin{ack}
The authors gratefully acknowledge support from the NSF under grants
DMS-0900239 and DMS-1302050. 
\end{ack}


\vfill


\vfill

\end{document}